\documentclass[12pt, reqno]{amsart}

\usepackage{amssymb, amsmath, amsthm}
\usepackage[backref]{hyperref}
\usepackage[alphabetic,backrefs,lite]{amsrefs}
\usepackage{amscd}   
\usepackage{fullpage}
\usepackage[all]{xy} 

\DeclareFontEncoding{OT2}{}{} 

\usepackage{color}

\newtheorem{lemma}{Lemma}[section]
\newtheorem{theorem}[lemma]{Theorem}

\newtheorem{prop}[lemma]{Proposition}
\newtheorem{cor}[lemma]{Corollary}

\newtheorem{claim*}{Claim}

\newtheorem{remark}[lemma]{Remark}
\newtheorem{thm}[lemma]{Theorem}
\newtheorem{defn}[lemma]{Definition}

\newtheorem{notation}[lemma]{Notation}
\newtheorem{assumption}[lemma]{Assumption}

\newcommand{\PP}{{\mathbb P}}
\newcommand{\C}{{\mathbb C}}
\newcommand{\F}{{\mathbb F}}
\newcommand{\Q}{{\mathbb Q}}
\newcommand{\R}{{\mathbb R}}
\newcommand{\Z}{{\mathbb Z}}

\newcommand{\Fbar}{{\overline{\F}}}

\newcommand{\pp}{{\mathfrak p}}

\newcommand{\ve}{{\varepsilon}}

\newcommand{\calF}{{\mathcal F}}

\newcommand{\calO}{{\mathcal O}}
\newcommand{\calP}{{\mathcal P}}
\newcommand{\calQ}{{\mathcal Q}}

\newcommand{\calS}{{\mathcal S}}

\newcommand{\calU}{{\mathcal U}}

\newcommand{\fraka}{{\mathfrak a}}
\newcommand{\Cbar}{{\mathbb C}}

\DeclareMathOperator{\rk}{rk}

\DeclareMathOperator{\ord}{ord}
\DeclareMathOperator{\Sym}{Sym}

\DeclareMathOperator{\Jac}{Jac}

\DeclareMathOperator{\Trop}{Trop}
\DeclareMathOperator{\trop}{trop}

\DeclareMathOperator{\vertt}{vert}

\DeclareMathOperator{\New}{New}
\DeclareMathOperator{\MV}{MV}

\DeclareMathOperator{\vol}{vol}
\DeclareMathOperator{\Sp}{Sp}

\DeclareMathOperator{\conv}{conv}
\DeclareMathOperator{\codim}{codim}
\DeclareMathOperator{\Alb}{Alb}
\DeclareMathOperator{\rank}{Rank}
\DeclareMathOperator{\Per}{Per}
\DeclareMathOperator{\red}{red}

\newcommand{\isom}{\cong}

\numberwithin{equation}{section}
\numberwithin{table}{section}

\title{Effective Chabauty for symmetric powers of curves}
\author{Jennifer Park}
\curraddr{Department of Mathematics, University of Michigan, Ann Arbor, MI 48109, USA}
\email{jmypark@umich.edu}
\date{\today}

\begin{document}
\begin{abstract}
Faltings' theorem states that curves of genus $g \geq 2$ have finitely many rational points. Using the ideas of Faltings, Mumford, Parshin and Raynaud, one obtains an upper bound on the number of rational points (see \cite{Szp85}, XI, \S 2), but this bound is too large to be used in any reasonable sense. In 1985, Coleman showed \cite{Col85} that Chabauty's method, which works when the Mordell-Weil rank of the Jacobian of the curve is smaller than $g$, can be used to give a good effective bound on the number of rational points of curves of genus $g > 1$.  We draw ideas from nonarchimedean geometry to show that we can also give an effective bound on the number of rational points outside of the special set of $\Sym^dX$, where $X$ is a curve of genus $g > d$, when the Mordell-Weil rank of the Jacobian of the curve is at most $g-d$. 
\end{abstract}

\maketitle

\section{Introduction}

Throughout the paper, we assume that $X$ is a \textbf{nice} (smooth, projective, and geometrically integral) curve of genus $g$ defined over $\Q$ that has a rational point $O \in X(\Q)$, and $d \geq 1$. We aim to generalize the following theorem of Coleman to $\Sym^dX$, the symmetric powers of curves:

\begin{theorem}[\cite{Col85}, Theorem 4]
\label{T: Coleman}
\label{T: Coleman1}
Let $g > 1$ and $p$ be a prime number. Then there is an effectively computable bound $N(g,p)$ such that for every nice curve $X$ defined over $\Q$ of good reduction at $p$, such that $X$ is of genus $g$ and $g > \rank (\Jac(X))(\Q)$, then
\[
\#X(\Q) \leq N(g,p).
\]
\end{theorem}

Although this theorem is weaker than Faltings' theorem for curves, which states that \textit{any} curve of genus $g \geq 2$ has finitely many rational points, Coleman's bounds are effective and sometimes sharp, in which case Theorem \ref{T: Coleman} can be used to find all rational points of a given curve. Previously, all known bounds were too large to be practical.

Coleman divides the set $X(\Q_p)$ into finitely many sets called \textit{residue disks}; the set of $\Q_p$-points on each residue disk is in bijection with $p\Z_p$. On each residue disk, a necessary condition for the $\Q_p$-points of $X$ (considered as an element of $p\Z_p$) to be in $X(\Q)$ is given as a power series equation; each $\Q$-point is a solution to the power series equation. The number of such solutions can be estimated by using Newton polygons.

More generally, consider $\Sym^dX$. While it seems plausible that one could generalize Chabauty's method to $\Sym^dX$, several problems exist (see \S \ref{S: Chabauty} for more explanation). One major such problem is the fact that $(\Sym^dX)(\Q)$ is not necessarily finite. However, in such cases, all but finitely many rational points of $\Sym^dX$ are contained in the \textit{special set}:

\begin{defn}[\cite{Lan91}]
\label{D: SpecialSet}
Let $X/\Q$ be a projective variety considered as a variety defined over $\overline{\Q}$. The \textbf{special set} of $X$ is the Zariski closure of the union of all images of nonconstant rational maps $f: G \to X$ of group varieties $G$ into $X$; these rational maps may be defined over finitely generated extensions of $\Q$. We denote the special set of $X$ by $\calS(X)$ (despite the name, this is a geometric object).
\end{defn}

As in \cite{Col85}, we will find locally analytic functions (written as power series on residue disks) whose solutions contain $(\Sym^dX)(\Q)$, defined by $p$-adic integrals. These power series cut out a rigid analytic subvariety of $(\Sym^dX)^{\textup{an}}$, denoted $(\Sym^dX)^{\eta=0}$ (Here, $(\Sym^dX)^{\textup{an}}$ denotes the analytification of $\Sym^dX$, in the sense of rigid analytic geometry).
Then $(\Sym^dX - \calS(\Sym^dX))(\Q)$ is contained in the set 
\[
\{P \in (\Sym^dX)^{\eta=0} : P \textup{ is the point in a $0$-dimensional component in }(\Sym^dX)^{\eta=0}\},
\]
under the following assumption (which always holds if $\rank J \leq 1$):

\begin{assumption}
\label{A: assumption}
Every positive-dimensional rigid analytic component of $(\Sym^dX)^{\eta=0}$ is contained in $\calS(\Sym^dX)^{\textup{an}}$. 
\end{assumption}

This assumption always holds if $\rk J \leq 1$, since we can always choose the locally analytic functions so that $(\Sym^dX)^{\eta = 0} \subseteq \overline{J(\Q)}^{p\textup{-adic}}$, where $\overline{J(\Q)}^{p\textup{-adic}}$ is at most $1$-dimensional $p$-adic Lie group.

The main result of this paper is the following; under Chabauty-type assumptions, as well as the above assumption, one can get an effective upper bound on the number of rational points of $\Sym^dX$ outside of the special set:

\begin{theorem}
\label{T: main}
Let $d \geq 1$, $p$ a prime, and $g \geq 2$. Then there exists a number $N(p,d,g)$ that can be computed effectively, such that for every nice curve $X$ defined over $\Q$ of good reduction at $p$ with $\rank J \leq g-d$ satisfying Assumption \ref{A: assumption}, 
\[
\#\{Q \in (\Sym^dX)(\Q) \mid \textup{$Q$ does not belong to the special set}\} \leq N(p,d,g).
\]
\end{theorem}

If we impose extra conditions on the above theorem, we can even get a better bound on $N(p,d,g)$. For example:

\begin{cor}
We can take $N(2,3,3) = 1539$ for any degree $7$ odd hyperelliptic curve $X$ such that $\rank J(\Q) \leq 1$ and such that $X$ has good reduction at $2$. 
\end{cor}

The problem of rational points on symmetric powers of curves have been studied in several papers. One result is that of Debarre and Klassen \cite{DebKla94}, which studies the \textup{Fermat curves} (projective plane curves given by
$X^N + Y^N = Z^N, N \geq 4$).
By Fermat's Last Theorem, we already know that these curves only have finitely many $K$-points for any number field $K$, and no nontrivial $\Q$-points, \cite{DebKla94} uses geometric methods to prove the following theorem:

\begin{thm}[\cite{DebKla94}]
For $N \neq 6$, there are only finitely many number fields $K$ with degree $d = [K:\Q] \leq N-2$ such that $F_N(K) \neq F_N(\Q)$.
\end{thm}

\cite{DebKla94} raises the question of applying Chabauty's method to symmetric powers of curves. Then \cite{Kla93} attempts to generalize \cite{Col85} to symmetric powers of curves:

\begin{thm}[\cite{Kla93}] Let $1 < d < \gamma$, and let $X$ be a nice curve of genus $g >2$ and gonality $\gamma$, satisfying $\rank J(\Q) \leq g-d$. Then there exists a canonical divisor $M$ on $(\Sym X^d)_{\Q_p}$ such that the complement $\Sym^dX \backslash M$ has only finitely many rational points (here, a canonical divisor is a divisor of a meromorphic $d$-form). Further,
\[
\#((\Sym^dX)(\Q) \backslash \red_p^{-1}(\bar M(\F_p))) \leq \#((\Sym^dX)(\F_p) \backslash \bar M(\F_p)),
\]
where $\red_p$ denotes the reduction modulo $p$ map, and $\bar M$ denotes the reduction of $M$ $\bmod \textup{ } p$.
\end{thm}

Also, \cite{Sik09} refines \cite{Kla93} by removing the gonality from the hypothesis of the above statement, and also giving a sufficient criterion for when a residue disk contains a single rational point. Also, he developed a method that can be used to compute $(\Sym^2X)(\Q)$ for some curves (two explicit examples are worked out in \cite{Sik09}, \S6).

Further, the components of the special set contained in the symmetric powers of curves have been studied, e.g. in \cite{HarSil91}, for the case of $d=2$. As $d$ grows, the geometry becomes increasingly complicated, as in \cite{Abr91}.

\begin{theorem}[\cite{HarSil91}]
If $\Sym^2X$ contains an elliptic curve, then $X$ is either bielliptic or hyperelliptic.
\end{theorem}

Two main ideas are required to obtain an effective bound for the number of points outside of the special set of $\Sym^dX$. The first is the approximation of the shape of the \textit{generalized} Newton polygons of multivariate power series, which gives an upper bound on the number of zeros of the power series equations on residue disks. In general, approximating the shape of the Newton polygons of multivariate power series is hard, but in our case, we have:

\centerline{
\xymatrix{
&& X^d(\C_p) \ar@{>>}[d]\\
\calU \ar@{^{(}->}[r] & (\Sym^dX)(\Q_p)  \ar@{^{(}->}[r] \ar@{>>}[d] & (\Sym^dX)(\C_p)\ar[r]\ar@{>>}[d] & J  \\
& (\Sym^dX)(\F_p)\ar@{^{(}->}[r] & (\Sym^dX)(\overline{\F_p}) & \\
}
}
While Chabauty's method typically deals with residue disks $\calU \subseteq (\Sym^dX)(\Q_p)$, we will instead look at the preimage of the residue disk $\calU \subseteq (\Sym^dX)(\Q_p) \subseteq (\Sym^dX)(\C_p)$ in $X^d(\C_p)$, where $\calU$ decomposes into a product of $d$ one-dimensional residue disks $\calU_i \subseteq X(\C_p)$ above $P_i \in X(\overline{\F_p})$. Pullbacks of residue disks $\calU$ to $X^d(\C_p)$ has the effect of change of variables on the local coordinates of $\calU$ into the uniformizing parameters of $X^d(\C_p)$, and this writes the multivariate power series constraints as a sum of $d$ single-variable power series (one variable for each power series) over a degree $d$ extension of $\Q_p$. Then the Newton polygons are much easier to approximate.

However, if $\#(\Sym^dX)(\Q) = \infty,$ we wish to count only the rational points outside of the special set, while the $d$ power series equations in $d$ variables have infinitely many common zeros. Under Assumption~\ref{A: assumption}, all such rational points form a subset of the set of zero-dimensional components of $(\Sym^dX)^{\eta=0}$. And such components correspond to the stable intersections of the multivariate power series constraints, whereas the positive-dimensional components of $(\Sym^dX)^{\eta=0}$ do not. Thus, we use deformation theory techniques coming from rigid analytic geometry, to deform the power series away from one another to obtain finite intersection in this case. This finite intersection number corresponds to the upper bound on the points outside of the special set.

In \S\ref{S: Chabauty}, we provide an outline of Chabauty's method for symmetric powers of curves. However, there is an intrinsic difficulty to Chabauty's method that comes from the incongruity between the algebraic and analytic description of the rational points in $\Sym^dX$, necessitating Assumption~\ref{A: assumption}. This problem is explained in \S\ref{S: comparison}. Then \S\ref{S: padicgeometry} will define generalized Newton polygons and the approximation of the shape the Newton polygons, and combine this idea with \S\ref{S: Chabauty} to obtain an upper bound for the number of points outside of the special set, when $(\Sym^dX)^{\eta=0}$ consists of only zero-dimensional components. We deal with the general case in \S\ref{S: deformation}, where we explain the idea of small $p$-adic deformations. Finally, in \S\ref{S: application}, we obtain some consequences of having an effective bound for the number of rational points outside of the special set of $\Sym^dX$.

\section*{Acknowledgements}
I would like to thank my advisor, Bjorn Poonen, for introducing me to Chabauty's method, for many helpful conversations on this project, and for his feedback on the exposition of this article. This project became the topic of my PhD thesis at MIT. I also thank Joseph Rabinoff for patiently explaining many of his results to me; many results in Section \ref{S: padicgeometry} of this paper were built on his results. I also benefited from conversations with Matt Baker, Jennifer Balakrishnan, Eric Katz, Samir Siksek, Bernd Sturmfels, and David Zureick-Brown.


\section{Chabauty on $\Sym^dX$}
\label{S: Chabauty}

In this section, we consider the problem of counting rational points outside of the special set of $\Sym^dX$. Using Chabauty's method, we will reduce this problem to analyzing the common zeros of $d$ power series in $d$ variables, and the power series have specific forms.


\subsection{Classical Chabauty} 
The exposition in this subsection outlines the classical method of Chabauty that gives an upper bound on $\#X(\Q)$; what we have here is a summarized version of \cite{McCPoo10}.
Let $\iota$ be the $\Q$-embedding
\begin{align*}
\iota: X & \hookrightarrow J\\
P & \mapsto [P-O]
\end{align*}
where $J$ is the Jacobian of $X$, viewed as the group of linear equivalence classes of degree-zero divisors on $X$. Then $J$ is an abelian variety of dimension $g$ over $\Q$. By an abuse of notation, we denote $\iota(X)$ as $X$. The inclusion $X(\Q) \subseteq J(\Q)$ holds, since $O \in X(\Q)$. Since $X(\Q) \subseteq X(\Q_p)$, we have $X(\Q) \subseteq X(\Q_p) \cap \overline{J(\Q)}$, where $\overline{J(\Q)}$ denotes the $p$-adic closure of $J(\Q)$ inside $J(\Q_p)$. Chabauty's result is: 

\begin{thm}[\cite{Cha41}]
\label{T: Chabauty}
Keep the notation as above. Let $X$ be a curve of genus $g \geq 2$ over $\Q$. If $X$ has good reduction at a prime $p$ and if $\dim \overline{J(\Q)} < g$, then $X(\Q_p) \cap \overline{J(\Q)}$ is finite.
\end{thm}

To compute $\#(X(\Q_p) \cap \overline{J(\Q)})$ explicitly under the hypothesis of Theorem \ref{T: Chabauty}, let $J_{\Q_p}$ and $X_{\Q_p}$ be the base changes of $J$ and $X$ to $\Q_p$. There is a bilinear pairing
\begin{align*}
J(\Q_p) \times H^0(J_{\Q_p}, \Omega^1) &\to \Q_p \\
(Q, \omega) &\mapsto \int_{O}^Q \omega
\end{align*}
which induces the \textbf{logarithm homomorphism}
\[
\log: J(\Q_p) \to H^0(J_{\Q_p}, \Omega^1)^* = T_OJ_{\Q_p},
\]
where $T_OJ_{\Q_p}$ denotes the tangent space to $J_{\Q_p}$ at the origin $O$. Since $\dim \log(\overline{J(\Q)}) <  g$, there exists a hyperplane $H \subseteq T_OJ_{\Q_p}$ containing $\log(\overline{J(\Q)})$. This hyperplane $H$ is defined by the vanishing of some $\omega_J \in H^0(J_{\Q_p}, \Omega^1) \isom T_OJ_{\Q_p}^*$, and the restriction of $\omega_J$ to $X_{\Q_p}$ can be uniquely identified with an $\omega_X \in H^0(X_{\Q_p}, \Omega^1)$ via:

\begin{prop}[\cite{Mil86}, Proposition~2.2]
\label{P: Differentials on J and X are the same}
The restriction map $H^0(J_{\Q_p}, \Omega^1) \to H^0(X_{\Q_p}, \Omega^1)$ induced by $X \hookrightarrow J$ is an isomorphism of $\Q_p$-vector spaces.
\end{prop}

Then, the map induced from the above bilinear pairing using the $\omega_J$ given by
\begin{align*}
J(\Q_p) &\to \Q_p \\
Q &\mapsto \int_{O}^Q \omega_J
\end{align*}
vanishes on $\overline{J(\Q)}$ by construction, so $\# (X(\Q_p) \cap \overline{J(\Q)})$ is bounded above by the number of zeros of the restriction
\begin{align*}
\eta: X(\Q_p) &\to \Q_p\\
Q &\mapsto \int_O^Q \omega_X,
\end{align*}
where $\omega_X = \iota^* \omega_J$.

From a computational perspective, it is known that $\omega$ has a well-defined power series expansion in terms of a uniformizer $t$ on small enough open subgroups $U$ of $X(\Q_p)$. In fact, $U$ can be taken to be \textbf{residue disks} of the reduction $\red_p: X(\Q_p) \twoheadrightarrow X(\F_p)$, which are the preimages of any point $Q \in X(\F_p)$. Such residue disk $U$ can be parametrized by a uniformizer $t$, which gives a set bijection $t: U \to p\Z_p$. Then one expresses the locally analytic function $\eta$ as a power series in terms of $t$ on $U$; the local coordinates can be chosen so that $\omega(t) \in \Z_p[[t]]$, and the number of zeros of $\eta$ on each residue disk can then be estimated using Newton polygons.


\subsection{Chabauty on $\Sym^d X$}
\label{SS: Symd}

In theory, it seems plausible that Chabauty's method could still apply to any higher-dimensional variety $Y$, where the Albanese variety $\Alb(Y)$ is used in place of the Jacobian, look for all rational points on the image of the Albanese map using a similar technique. However, there exist several problems in generalizing Chabauty's method to arbitrary higher-dimensional varieties. 

\begin{itemize}
\item[(1)] $\Alb(Y)$ may be trivial: since $\dim \Alb(Y) = h^0(Y, \Omega_1)$, if $h^0(Y, \Omega_1) = 0$, then Chabauty's method yields nothing. For example, if $Y$ is a K3 surface or an Enriques surface, $h^{0,1} = 0$, so Chabauty's method cannot apply.

\item[(2)] To understand $Y(\Q)$, we need to understand the (rational points of the) fibres of $j$ as well as the rational points of $j(Y)$. Understanding the fibres may be complicated.

\item[(3)] The case of higher-dimensional varieties allows for the possibility that $\#Y(\Q) = \infty$: for example, if $Y = \Sym^2 X$ for a hyperelliptic curve $X: y^2 = f(x)$, then $\{(t, \sqrt{f(t)}),(t, -\sqrt{f(t)})\} \in (\Sym^2X)(\Q)$ for all $t \in \Q$.
\end{itemize}

The first two problems are taken care of by choosing $Y = \Sym^d X$, where $X$ is a nice curve of sufficiently high genus $g$ (to be chosen later); then $\Alb(Y) = \Jac(X)=:J$, so the Albanese variety is nontrivial, and understanding the fibres of the Albanese map
\begin{align*}
j: (\Sym^dX)(\Q) &\to J(\Q)\\
\{P_1, \ldots, P_d\} &\mapsto [P_1 + \cdots + P_d - d\cdot O],
\end{align*}
is not too difficult:
\begin{lemma}
\label{L: ProjSpace}
Suppose that $Q \in (\Sym^dX)(\Q)$. Then the set of rational points on the fibre of $j$ containing $Q$ is isomorphic to $\PP^n(\Q)$ for some $n \geq 0$.
\end{lemma}
\begin{proof}
As $J$ parametrizes the equivalence classes of degree-$0$ divisors, $Q$ is identified with an effective divisor on $X$. By \cite{Har77}*{Theorem~II.5.19}, the set of points giving rise to the same divisor is isomorphic to a finite-dimensional vector space. In particular, if a fibre contains two distinct rational points, then dimension of this vector space is at least $1$.
\end{proof}
To deal with the last problem, we recall from \cite{Fal94}:
\begin{theorem}[\cite{Fal94}]
\label{T: Faltings}
Let $A/\Q$ be an abelian variety, and $X \subseteq A$ be a closed subvariety. Then there exist finitely many subvarieties $Y_i \subset X$ such that each $Y_i$ is a coset of an abelian subvariety of $A$ and
\[
X(\Q) = \bigcup Y_i(\Q).
\]
\end{theorem}

Apply this theorem to the image of $j$ to get $j((\Sym^dX)(\Q)) = \bigcup_{\textup{finite}} Y_i(\Q)$. From Lemma~\ref{L: ProjSpace} and Theorem~\ref{T: Faltings}, we see that there are two ways of obtaining $\#(\Sym^dX)(\Q) = \infty$: either at least one of the fibres of $j$ is nontrivial, or there exists some $Y_i$ with $\dim Y_i > 0$. Excluding the $\Q$-points on $\Sym^dX$ accounted by these two possibilities, we are left with finitely many rational points on $\Sym^dX$. However, to avoid ambiguity coming from Faltings' theorem, one excludes the special set (Definition~\ref{D: SpecialSet}) instead, which includes both of the aforementioned possibilities.

For the rest of the section, assume that $X$ satisfies $\rk J \leq g-d$. Further, let $p$ be a prime, and let $X$ have good reduction at $p$. We observe that
\[
(\Sym^dX)(\Q) \subseteq j^{-1}(j(\Sym^dX)(\Q_p) \cap \overline{J(\Q)}).
\]
We obtain locally analytic functions come from integrating $\omega_i \in H^0(J_{\Q_p}, \Omega^1)$, as in the previous section. However, we use the following stronger definition:

\begin{defn}
For any $\omega_J \in H^0(J_{K}, \Omega^1)$, define the map $\eta: J(\overline{\Q}_p) \to \overline{\Q}_p$ by taking the inverse limit of the maps
\begin{align*}
\eta_K: J(K) &\to K\\
Q &\mapsto \int_{O}^Q \omega_J,
\end{align*}
for each $p$-adic field $K$. 
\end{defn}

The $p$-adic integrals also satisfy for $Q_1, Q_2 \in J(\overline{\Q}_p)$ and $\omega_J \in H^0(J_{\Q_p}, \Omega^1)$,
\[
\int_{O}^{Q_1+Q_2}\omega_J = \int_{O}^{Q_1}\omega_J + \int_{Q_1}^{Q_1 + Q_2}\omega_J = \int_{O}^{Q_1}\omega_J + \int_{O}^{Q_2}\omega_J,
\]
where the first equality follows from linearity, and the second equality follows from the translation-invariance of $p$-adic integrals. Then we may define the integral on $\Sym^dX$ via the pullback of the integral on $W_d \subseteq J$, which can be written using the above as
\[
\int_{O}^{[P_1 + \cdots + P_d - d \cdot O]}\omega_J =  \int_O^{[P_1-O]} \omega_J + \cdots + \int_O^{[P_d-O]}\omega_J.
\]
Therefore, the corresponding locally analytic function on $\Sym^dX$ can be written as

\begin{align*}
\eta \colon (\Sym^dX)(\Q_p) &\to \Q_p\\
\{P_1,P_2, \ldots, P_d\} &\mapsto \eta_J([P_1+P_2+ \cdots + P_d-d\infty])\\
& = \int_O^{P_1} \omega_X + \int_O^{P_2} \omega_X +\cdots + \int_O^{P_d} \omega_X,
\end{align*}
where $\omega_X \in H^0(X_{\Q_p}, \Omega^1)$ is the pullback of $\omega_J$ via the isomorphism given in Proposition \ref{P: Differentials on J and X are the same}, and the $P_i$ are defined over some field $K$ with $[K:\Q] \leq d$.

Since $\dim \overline{J(\Q)} \leq g-d$, then there exist $(\omega_J)_1, (\omega_J)_2, \ldots, (\omega_J)_d \in H^0(J_{\Q_p}, \Omega^1)$ that are linearly independent, such that the corresponding locally analytic functions on $J_{\Q_p}$ obtained by integrating the $(\omega_J)_i$ vanish on $\overline{J(\Q)}$.  Let $\eta_1, \eta_2, \ldots, \eta_d$ be the locally analytic functions on $\Sym^dX$ corresponding to $(\omega_J)_1, (\omega_J)_2, \ldots, (\omega_J)_d$, respectively. Possibly $\eta_1, \eta_2, \ldots, \eta_d$ have infinitely many common zeros, which contain either the linear system of equivalent divisors parametrized by $\PP^n$ for some $n \geq 1$, or the (infinitely many) rational points coming from the rational points in the special set $(\calS(\Sym^dX))(\Q)$.

\subsection{Explicit parametrization of points on residue disks}
From \S\ref{SS: Symd}, we have (at least) $d$ nontrivial and independent locally analytic functions on $(\Sym^dX)(\Q_p)$ whose common zeros contain the $p$-adic points $j^{-1}((j(\Sym^dX))(\Q_p) \cap \overline{J(\Q)})$. We estimate the number of common zeros  of these analytic functions away from the special set of $\Sym^dX$. In order to do this explicitly, we work locally to get power series expansions.

Since $X$ has good reduction at $p$, $X$ is a smooth proper variety over $\Z_p$. This implies that $\Sym^dX$ is smooth and proper over $\Z_p$. By the valuative criterion for properness, $(\Sym^dX)(\Q_p) = (\Sym^dX)(\Z_p)$.

\begin{defn}
A \textbf{residue disk} $\calU$ of $\Sym^dX$ is the preimage of a point in $(\Sym^dX)(\Fbar_p)$ under the reduction modulo-$p$ map
\[
\red_{p}: (\Sym^dX)(\C_p) \twoheadrightarrow (\Sym^dX)(\Fbar_p).
\]
If $K$ is a finite extension of $\Q_p$, the set of $K$-points of the residue disk $\calU$ are defined to be the set $\calU \cap (\Sym^dX)(K)$, and these are denoted $\calU(K)$.
\end{defn}

Let $\calU$ be the residue disk over $O \in (\Sym^dX)(\Fbar_p)$. Then $\calU$ fits into the exact sequence
\[
0 \to \calU \to J(\C_p) = J(\calO_{\C_p}) \to J(\Fbar_p) \to 0,
\]
where the equality in the middle follows from the valuative criterion for properness.
Then for any finite extension $K$ of $\Q_p$, we have $\calU(K) = \{\calP \in (\Sym^dX)(K): \red_{p}(\calP) = \{P_1, \ldots, P_d\}\}$.  Thus, by a Hensel-type argument, there is a bijection
\[
(u_1, \ldots, u_d):  \calU(K) \stackrel{\sim}{\longrightarrow} (p_K \calO_K)^d
\]
between the set of $K$-points of the residue disk mapping to $\{P_1, \ldots, P_d\}$ via some local coordinates $u_1, \ldots, u_d$, where $p_K$ is the uniformizer of $K$.

In practice, this parametrization is not practical: \S \ref{SS: Symd} suggests that we write the higher-dimensional integrals in terms of several $1$-dimensional integrals expanded around various $\F_{p^n}$-points $P_i$.

\subsubsection{The case of $\Sym^2X$}
\label{SSS: Sym2}
For simplicity, we first consider the case when $d=2$. First suppose that we consider the residue disk $\calU(\Q_p)$ which consists of the $\Q_p$-points in $\Sym^2X$ reducing to $\{P, P\} \in (\Sym^2X)(\F_p)$ for some $P \in X(\F_p)$. The completion of the local ring of $(X \times X)_{\Q_p}$ near any pair of points $(Q_1, Q_2) \in X \times X$ reducing to $\{P, P\}$ is given by $\Q_p[[t_1, t_2]]$, where $t_1$ and $t_2$ denote the uniformizers for the set of $\Q_p$-points of the residue disks $\calU_1$ around $Q_1$ and $\calU_2$ around $Q_2$ in $X$, respectively. We further assume that $t_1$ and $t_2$ vanish at $Q_1'$ and $Q_2'$, respectively. This means that we have two bijections
\[
t_1: \calU_1 \stackrel{\sim}{\longrightarrow} p\Z_p, \quad t_2: \calU_2 \stackrel{\sim}{\longrightarrow} p\Z_p.
\]
Since $t_1(Q_1') = 0$, and $t_2(Q_2') = 0$, for any $\omega \in H^0(J_{\Q_p}, \Omega^1)$, the Coleman integral
\begin{align*}
\eta: \Sym^dX(\Q_p) &\to \Q_p\\
\{Q_1, Q_2\} &\mapsto \int_{O}^{\{Q_1, Q_2\}} \omega
\end{align*}
can be written as the following, in terms of the $t_i$:
\begin{align*}
\eta(\{Q_1, Q_2\}) &= \int_{O}^{\{Q_1, Q_2\}}\omega = \int_{O}^{Q_1}\omega + \int_{O}^{Q_2}\omega\\
&= \int_{O}^{Q_1'}\omega+ \int_{Q_1'}^{Q_1}\omega + \int_{O}^{Q_2'}\omega + \int_{Q_2'}^{Q_2}\omega\\
&= C + \int_0^{t_1(Q_1)}\omega(t_1) + \int_0^{t_2(Q_2)}\omega(t_2),
\end{align*}
where $C = \int_O^{Q_1'}\omega+\int_O^{Q_2'}\omega$ is a constant in $\Q_p$ (that depends on the choice of $Q_i'$).

One can relate the $t_i$ to the original local coordinates of $\calU$. If $r(Q_1, Q_2) = (P_1, P_2)$ in $(\Sym^dX)(\F_p)$, then the completion of the local ring at $\{Q_1, Q_2\} \in (\Sym^2X)(\Q_p)$ is given by
\[
\Q_p[[t_1, t_2]]^{S_2} = \Q_p[[u_1, u_2]],
\]
so one could choose the $t_i$ and the $u_i$ to satisfy $u_1 = t_1 + t_2$ and $u_2 = t_1t_2$. This means that for each point $\{Q_1, Q_2\} \in (\Sym^2X)(\Q_p)$, which corresponds bijectively to a unique pair $(u_1, u_2)$, there are two pairs $(t_1, t_2)$ that correspond to it.

On the other hand, if the residue disk $\calU$ were the preimage of $\{P_1, P_2\} \in (\Sym^2X)(\F_p)$ with $P_1 \neq P_2$ under the reduction by $p$ map, the situation is simpler, as we have the following description of the residue disk.
\begin{align*}
\calU &= \{\{Q_1, Q_2\}: Q_i \in X(W(\F_{p^2})) \textup{ reducing to } P_i \in X(\F_{p^2})\} \\
&\subseteq X(W(\F_{p^2})) \times X(W(\F_{p^2}))
\end{align*}
where $W(\F_{p^2})$ denotes the Witt ring of $\F_{p^2}$. Thus, there are the bijections
\[
(t_1, t_2): \calU \stackrel{\sim}{\longrightarrow} p\Z_p \times p\Z_p,
\]
possibly defined over some quadratic extension of $\Q_p$, where $\calU_i$ denotes the residue disk around $Q_i$. Choose the basepoints of the lifts $t_1(Q_1') = 0$, and $t_2(Q_2') = 0$.

Then, we can write for each $\omega \in H^0(X_{\Q_p}, \Omega^1)$
\begin{align*}
\eta(\{Q_1, Q_2\}) &= \int_{O}^{\{Q_1, Q_2\}}\omega = \int_{O}^{Q_1}\omega + \int_{O}^{Q_2}\omega\\
&= \int_{O}^{Q_1'}\omega+ \int_{Q_1'}^{Q_1}\omega + \int_{O}^{Q_1'}\omega + \int_{Q_2'}^{Q_2}\omega\\
&= C + \int_0^{t_1(Q_1)}\omega(t_1) + \int_0^{t_2(Q_2)}\omega(t_2),
\end{align*}
where the $\omega(t_i)$ are defined (under appropriate scaling) over the ring of integers of some quadratic extension of $\Q_p$, and $C = \int_O^{Q_1'}\omega + \int_O^{Q_2'}\omega$ is a constant in $\Z_p$ (that depends on the choice of $Q_1'$ and $Q_2'$)
In this case, each pair $(t_1, t_2)$ represents a point on $\Sym^2X$ exactly once.

In either cases, we are able to express the Coleman integral in the following form:

\begin{defn}
A power series $f \in K[[t_1, \ldots, t_n]]$ is said to be \textbf{pure} if each of its terms are of the form $Ct_i^N$, with $C \in K$ and $N \in \Z_{\geq 0}$. In particular, a pure power series does not contain any term that is a product of more than one variable.
\end{defn}

Our goal for the rest of the section is to find pure power series that are related to the locally analytic functions in $d$ variables that we obtain from Chabauty's method.

\subsubsection{The case of $\Sym^dX$}
In this section, we generalize \S\ref{SSS: Sym2}. More concretely, our aim is to express each $p$-adic integral obtained from a $\omega \in H^0(X_{\Q_p}, \Omega^1)$ (see \S \ref{SS: Symd}), which is a power series of each residue disk of $(\Sym^dX)(\Q_p)$, as a pure power series over some extension field of $\Q_p$ on the residue disk.  We will show that this is possible by doing a change of variables on the local coordinates of each residue disk. We fix a holomorphic differential $\omega$ from which we get one of the $d$ power series vanishing on $j^{-1}(\overline{J(\Q)} \cap j(\Sym^dX)(\Q_p))$ as in section \S \ref{SS: Symd}.

We now consider the residue disk given as the preimage of the point $\{P_1, \ldots, P_d\} \in (\Sym^dX)(\F_p)$ under the reduction map modulo $p$. The multiset $\{P_1, \ldots, P_d\}$ can be decomposed as the disjoint union of the multisets  of the form
\[
\calS := \{P_{i_1}, \ldots, P_{i_s}: P_{i_1} = \cdots = P_{i_s}, P_{i_1} \neq P_j \textup{ for } j \in \{1, \ldots, d\} - \{i_1, \ldots, i_s\}\},
\]
where $\{i_1, \ldots, i_s\} \subseteq \{1, \ldots, d\}$. Consider the locally analytic function obtained by $p$-adic integration with respect to an $\omega \in H^0(J_{\Q_p}, \Omega^1)$ on $(\Sym^dX)(\Q_p)$ given by

\begin{align*}
\eta_{\omega}: \left(\prod_{\calS}(\Sym^{\#\calS}X)\right)(\Q_p) &\to \Q_p\\
\prod_{\calS}(\{P_{i_1}, \ldots, P_{i_s}\}) & \mapsto \sum_{\calS} \left(\sum_{k=1}^{\#\calS} \int_O^{[P_{i_k} - O]}\omega\right). \\
\end{align*}

As in \S \ref{SSS: Sym2}, the terms corresponding to the different $\calS$ can be separated. So we consider the terms that depend on $\calS$ from the above expression; namely the terms $\sum_{k=1}^{\#\calS} \int_O^{[P_{i_k} - O]}\omega$. 

When $\#\calS = 1$, we can expand as in \cite{Col85}, but since $P_{i_1} \in X(\F_q)$ where $q = p^{\ell}$ for some $\ell \geq 1$, its expansion with respect to the uniformizer $t_1$ satisfies $\eta_{\calS, \omega} \in W(\F_q)[\frac{1}{p}][[t_1]]$. Here, $W(\F_q)$ denotes the Witt ring of $\F_q$, and $W(\F_q)[\frac{1}{p}]$ is the fraction field of the Witt ring. This is the degree-$p^{\ell}$ unramified extension of $\Q_p$, and the points in the residue disk of $P_{i_1}$ are parametrized by $t_1 \in pW(\F_q)$.

Now suppose that $\#\calS \geq 2$ and $P_{\calS} \in X(\F_q)$. Let $\calU$ be the residue disk in $(\Sym^{\#\calS}X)(\Q_p)$ reducing to $\{P_{\calS}, \ldots, P_{\calS}\}$ (the multiset where $P_{\calS}$ is repeated $\#\calS$ times). Let $\calU_{i_1}, \ldots ,\calU_{i_s}$ be the residue disks around $Q_{i_1}, \ldots, Q_{i_s}$ in $X(W(\F_q))$ with the set bijections
\[
t_{i_j}: \calU_{i_j} \stackrel{\sim}{\longrightarrow} pW(\F_q)
\]
for each $1 \leq j \leq s$, with $t_{i_j}(0) = Q_{i_j}$. Then the Coleman integral $\int_O^{Q_{i_1}, \ldots, Q_{i_s}} \omega$ can be written as
\[
C + \int_0^{t_{i_1}(Q_{i_1})} \omega + \cdots + \int_0^{t_{i_s}(Q_{i_s})} \omega
\]
where $C$ is a constant depending on the choice of the $Q_{i_j}$, and $\omega$ is scaled to have coefficients in $W(\F_q)$. The $t_{i_j}$ are related to the original local coordinates $(u_{i_1}, \ldots, u_{i_s})$ of $\calU$ by
\[
W(\F_q)\left[\frac{1}{p}\right][[t_{i_1}, \ldots, t_{i_s}]]^{S_s} = W(\F_q)\left[\frac{1}{p}\right][[u_{i_1}, \ldots, u_{i_s}]]
\]
so one can take the $u_{i_k}$ to be the $k$-th elementary symmetric polynomial in $t_{i_j}$. This means that for each point $\{Q_{i_1}, \ldots, Q_{i_s}\} \in (\Sym^sX)(W(\F_q)\left[\frac{1}{p}\right])$, which corresponds bijectively to a unique pair $(u_{i_1}, \ldots, u_{i_s})$, there are $s!$ choices for the $s$-tuple $(t_{i_j})_j$.

The above discussion leads to the following proposition:

\begin{prop}
\label{P: PureChabauty}
\hfill
\begin{itemize}
\item[(a)] A power series $\eta_{\omega} = \left(\prod_{\calS}(\Sym^sX)\right)(\Q_p) \to \Q_p$ obtained from $p$-adic integration can be re-written via a change of variables as a pure power series, whose coefficients are contained in some extension of $\Q_p$ of degree at most $d$.
\item[(b)] Suppose that one obtains $d$ power series $\eta_1, \ldots, \eta_d$ via Chabauty's method as outlined in the previous section, and that one rewrites these power series as $\eta'_1, \ldots, \eta'_d$, where $\eta'_i$ are pure power series obtained from part (a). Then there is a $N$-to-one correspondence between the common zeros of the $\eta_i$ and $\eta'_i$, where $N = \prod_{\calS} (\#\calS)!$. Further, the solutions to $\eta_i'$ that correspond to the points $\Sym^dX(\Q_p)$ have $p$-adic valuations of at least $1/d$.
\end{itemize}
\end{prop}

Now, it remains to associate Newton polygons to these power series, and apply arguments analogous to \cite{Col85} to try to count the common zeros.

\section{Comparison of the algebraic loci and the analytic loci on $\Sym^dX$}

We are interested in comparing different sets that contain $(\Sym^dX)(\Q)$ inside $(\Sym^dX)(\Cbar_p)$. The different subsets of $(\Sym^dX)(\Cbar_p)$ that we consider are described below:
\begin{itemize}
\item[(i)] Faltings' theorem says that given the natural embedding $j: \Sym^dX \hookrightarrow J$ using the basepoint $O \in X$,
\[
j(\Sym^dX)(\Q) = \bigcup_{\textup{finite}} Y_i(\Q),
\]
where the $Y_i$ are cosets of abelian subvarieties of $J$ with $Y \subseteq W_d$. The set of points that we are interested in is the set of $\C_p$-points of the inverse image $\bigcup j^{-1}(Y_i)$, denoted $\calF(\Sym^dX)(\Cbar_p)$. We note that $\calF(\Sym^dX)(\C_p)$ depends on the choice of the $Y_i$.

\item[(ii)] The set of $\Cbar_p$-points of the special set $\calS(\Sym^dX)$ of $\Sym^dX$: recall that the special set was defined in Definition \ref{D: SpecialSet}. This set will be denoted $\calS(\Sym^dX)(\Cbar_p)$.

\item[(iii)] The set
\[
\{P \in (\Sym^dX)(\Cbar_p): \eta_i(j(P)) = 0 \textup{ for all } 1 \leq i \leq d\},
\]
where the $\eta_i$ are $d$ independent locally analytic functions on $J(\C_p)$ that vanish on $\overline{J(\Q)}$ arising from Chabauty's method. Since this definition depends on the choice of the $\eta_i$; we will fix one such choice here. We denote this set by $(\Sym^dX)^{\eta=0}$. 
\end{itemize}

In this section, we relate these different sets. 

If $d=1$, $g \geq 2$ and $\rk J < g$, then all of the above sets are zero-dimensional, which makes the comparison simple: We have $\emptyset = \calS(X)(\Cbar_p) \subseteq \calF(X)(\Cbar_p) \subseteq (X)^{\eta=0}$.

For $d>1$, we will see that these sets do not obey a linear containment relation; in particular, there does not seem to be any inclusion relation between $\calS(\Sym^dX)$ and $(\Sym^dX)^{\eta=0}$; this necessitates an extra technical hypothesis of Assumption \ref{A: assumption} to force such an inclusion. This seems to be an intrinsic limitation of Chabauty's method on higher-dimensional varieties; a new idea seems to be necessary to obtain more precise information on the rational points of $\Sym^dX$. We review some basics and terminology of rigid analytic geometry in $\S \ref{S: RAG}$ that will enable the comparison of the sets above in $\S \ref{S: comparison}$. 

\subsection{$(\Sym^dX)^{\eta=0}$ as a rigid analytic space}
\label{S: RAG}
We view $(\Sym^dX)^{\eta=0}$ as a rigid analytic space, whose admissible cover by affinoid spaces are given by the vanishing of certain Coleman integrals on residue disks, as in \S \ref{SS: Symd} (which are elements of the Tate algebra over $\Q_p$ with $d$ variables). For the basic terminology, we refer the readers to \cite{Con08} and \cite{Rab12}.

We note that there is a notion of irreducible components on rigid analytic spaces. The theory of irreducible components was first suggested in \cite{ColMaz98}, and simplified in \cite{Con99}. We summarize \cite{Con99} here:

\begin{defn}
A rigid analytic space $X$ is \textbf{disconnected} if there exists an admissible open covering $\{U,V\}$ of $X$ with $U \cap V = \emptyset$, where $U,V \neq \emptyset$. Otherwise, $X$ is said to be \textbf{connected}.
\end{defn}

\begin{defn}
Let $X$ be a rigid analytic space that admits a cover of affinoid spaces $\{\Sp A_{\lambda}\}_{\lambda \in \Lambda}$. A morphism $\pi: \widetilde X \to X$ is said to be a \textbf{normalization} if it is isomorphic to the morphism obtained by gluing $\Sp(\widetilde A_{\lambda}) \to \Sp A_{\lambda}$, where $\widetilde{A_{\lambda}}$ denotes the normalization of $A$ in the usual sense.
\end{defn}

It is known that for any rigid analytic space $X$, we can find a normalization $\pi: \widetilde X \to X$; for example, see \cite{Con99}*{Theorem~1.2.2}.

\begin{defn}[\cite{Con99}, Definition 2.2.2] 
The \textbf{irreducible components} of a rigid analytic space $X$ are the images of the connected components $X_i$ of the normalization $\widetilde X$ under the normalization map $\pi: \widetilde X \to X$.
\end{defn}

\begin{remark}
When $X = \Sp(A)$ is affinoid, the irreducible components of $X$ are the analytic sets $\Sp(A/\pp)$ for the finitely many minimal prime ideals $\pp$ of the noetherian ring $A$.
\end{remark}


\subsection{Comparing algebraic and analytic loci in $\Sym^dX$}
\label{S: comparison}
Now we determine the containment relations of the three sets mentioned at the beginning of this chapter; namely, $\calF(\Sym^dX), \calS(\Sym^dX)$, and $(\Sym^dX)^{\eta=0}$, as well as $(\Sym^dX)(\Q)$.

\begin{lemma}
\label{L: integral}
For any smooth projective curve $X$ with the choice of a rational point $O \in X(\Q)$ with good reduction at $p$, one has
$(\Sym^dX)(\Q) \subseteq (\Sym^dX)^{\eta=0}.$
\end{lemma}
\begin{proof}
By construction, each $\eta_i$ mentioned in part (iii) at the beginning of this section satisfies $\eta_i(P) = 0$ for all $P \in J(\Q)$. Thus, $(\Sym^dX)(\Q) \subseteq j^{-1}(J(\Q)) \subseteq (\Sym^dX)(\C_p)^{\eta=0}$.
\end{proof}

\begin{lemma}
\label{L: specialset}
We keep the notation of $Y_i$ and $\calS(\Sym^dX)(\Cbar_p)$ from the beginning of this section. Let $Y=Y_i$ for some $i$ such that $\dim Y > 0$. Then $j^{-1}(Y(\Cbar_p)) \subseteq \calS(\Sym^dX)(\Cbar_p)$.
\end{lemma}
\begin{proof}
We consider two cases: if the generic point of $Y$ has a positive-dimensional preimage, then each $Q \in Y(\Cbar_p)$ is $\PP^n$ for some $n > 0$, so each fibre is contained in $\calS(\Sym^dX)(\Cbar_p)$. On the other hand, if the preimage of the generic point of $Y$ is $0$-dimensional, then any irreducible component of $j^{-1}(Y)$ is either covered by positive-dimensional projective spaces, or is birational to $Y$ via the restriction of $j$. All of these irreducible components are then in the special set.
\end{proof}

In particular, we note that $\calF(\Sym^dX) \subseteq \calS(\Sym^dX)$. Finally, it remains to relate $\calS(\Sym^dX)(\Cbar_p)$ and $(\Sym^dX)^{\eta=0}$. In general, neither set is contained in the other; however, with Assumption \ref{A: assumption}, we immediately get:

\begin{prop}
\label{P: comparison}
Let $R_1, \ldots, R_n$ be the irreducible components of the rigid analytic space $(\Sym^dX)^{\eta=0}$. Further, suppose that $\Sym^dX$ satisfies Assumption \ref{A: assumption}. Then for each $R_i$ with $\dim R_i \geq 1$, we have
$R_i(\Cbar_p) \subseteq \calS(\Sym^dX)(\Cbar_p)$, and so
$(\Sym^dX)^{\eta=0} \subseteq \calS(\Sym^dX)$.
\end{prop}

Then taking complements of the relation obtained in Proposition \ref{P: comparison} inside $\Sym^dX$ and looking at the $\Q$-points, one obtains:

\begin{cor}
\label{cor: comparison}
Under the hypothesis of Proposition \ref{P: comparison}, one has
\[
\{\Q\textup{-points of } \Sym^dX \backslash \calS(\Sym^dX)\} \subseteq \bigcup \{0 \textup{-dimensional } R_i\}.
\]
\end{cor}

Thus, under the hypothesis of Proposition \ref{P: comparison}, one is still able to interpret the results given from Chabauty's method for higher-dimensional varieties, as Chabauty's method gives an upper bound on $\bigcup (\textup{0-dimensional } R_i)$.  
For the rest of the paper, we assume that the conditions of Proposition \ref{P: comparison} hold for $\Sym^dX$.


\section{$p$-adic geometry}
\label{S: padicgeometry}

The goal of this section is to associate a ``generalized Newton polygon" to each multivariate power series, and to state an approximation theorem for the number of roots of a system of equations given by $d$ power series in $d$ variables in general position -- that is, having finitely may common zeros -- using these Newton polygons. The classical case of $d=1$ is well-known in the literature. To define the Newton polygons for multivariate power series, we review the language necessary to define tropical objects, and state the results in tropical geometry. For a more detailed treatment of tropical geometry, see \cite{Mac13} and \cite{Rab12}.

\subsection{Tropicalization of rigid analytic hypersurfaces}
Tropical geometry generalizes the theory of Newton polygons of single-variable power series to power series of several variables. Most of the exposition from this section is taken from \cite{Rab12}.
We discuss the tropicalization of affinoid hypersurfaces cut out by a power series over $\C_p$ that arise via Coleman integration. These power series are necessarily convergent on the domain where each of the coordinates has a positive valuation (this has to do with the fact that we can write any $\omega \in H^0(X_{\Q_p}, \Omega^1)$ as a power series with coefficients in $\Z_p$). The tropicalization should be seen as the dual of a Newton polygon; this notion will be made precise in this section. More generally, everything in this section works for a nontrivially valued field $K$ that is complete with respect to the nontrivial valuation. We further assume that $K$ is algebraically closed.

\begin{defn}

Let $m = (m_1, \ldots, m_d) \in \Q_{\geq 0}^d$, and let $P_m := \{(x_1, \ldots, x_n) \in N_{\R} : x_i \geq m_i \textup{ for } 1 \leq i \leq d\}$. Let
\[
U_{P_m} = \{(x_1, \ldots, x_d): v(x_i) \geq m_i \textup{ for all } i\}.
\]

The \textbf{tropicalization map on $U_{P_m}$} is
\begin{align*}
\trop: U_{P_m}^d & \to (P_m)^d\\
(\xi_1, \ldots, \xi_d) &\mapsto (v(\xi_1), \ldots, v(\xi_d)).
\end{align*}
\end{defn}

The above tropicalization makes sense for affinoid hypersurfaces cut out by power series convergent on some $U_{P_m}$, i.e. the power series that are convergent when evaluated on some $x \in \Q_p^d$ with $v(x) \in P_m$. The set of such power series will be denoted by
\[
K \langle U_{P_m} \rangle := \left\{\sum_{u \in \Z_{>0}^d}a_ux^u : a_u \in K, v(a_u) + \langle u,w \rangle \to \infty \textup{ for all } w \in P_m\right\},
\] 
where the convergence $v(a_u) + \langle u,w \rangle \to \infty$ holds as $u$ ranges over $\Z_{>0}^d$ in any order. For example, if $m = (0, \ldots, 0)$, then
\[
K \langle U_{P_m} \rangle = K \langle x_1, \ldots, x_d \rangle,
\]
the Tate algebra in $d$ variables.

\begin{remark}
\label{R: CM}
More generally, it is known that $K \langle U_{P_m} \rangle$ is a $K$-affinoid algebra (\cite{Rab12}*{Lemma~6.9(i)}), a Cohen-Macaulay ring (\cite{Rab12}*{Lemma~6.9(v)}), and that $U_P = \Sp K \langle U_P \rangle$.
\end{remark}

\begin{defn}
Let $P$ be a polyhedron, and let $f_1, \ldots, f_n \in K \langle U_P\rangle$. Let $(f_1, \ldots, f_n)$ be the ideal in $K \langle U_P \rangle$ generated by $f_1, \ldots, f_n$. Then
\[
V(f_1, \ldots, f_n) := \Sp K \langle U_P \rangle/(f_1, \ldots, f_n).
\]
Then $V(f_1, \ldots, f_n)$ is an affinoid subspace of $\Sp K \langle U_P \rangle$.
\end{defn}

In our case, each the power series $f \in K[[x_1, \ldots, x_d]]$ that arises from Chabauty's method on $\Sym^dX$ converges when $v(x_i) > 0$ for $1 \leq i \leq d$. Thus, $f \in K\langle U_{P_m} \rangle$ for any $P_m$ with $m \in \Q_{>0}^d$. Let $f \in K \langle U_{P_m} \rangle$. Now we define $\Trop(f)$, the tropical variety corresponding to $f$, and then outline the procedure for computing $\Trop(f)$ in $\S$ \ref{S: Newton}.

\begin{defn}
For $f \in K \langle U_{P_m} \rangle$, 
\[
\Trop(f) := \overline{\trop(V(f))} = \overline{\{(v(\xi_1), \ldots, v(\xi_d)) : f(\xi) = 0, \xi \in U_{P_m}\}},
\]
where $V(f)$ is the affinoid subspace defined by the ideal $\fraka = (f) \subset K \langle U_{P_m} \rangle$. Here, we take the topological closure in $P_m$. 
\end{defn}

Often, the easiest way to compute $\Trop(f)$ is by using Lemma \ref{L: Kapranov}, which requires these definitions:

\begin{defn}
For $0 \neq f \in K \langle U_{P_m} \rangle$, write $f = \sum_{u \in S_{\sigma}}a_ux^u$. The \textbf{height graph of $f$} is
\[
H(f) = \{(u, v(a_u)): u \in \Z_{\geq 0}^d, a_u \neq 0\} \subseteq \Z_{\geq 0}^d \times \R.
\]
Given $w \in \Q_{>0}^d$, we also define
\[
m_f(w) = m(w) = \min_{u \in S_{\sigma}}\{(-w, 1) \cdot H(f)\},
\]
where $\cdot$ denotes the usual dot product, and
\[
\vertt_w(f) = \{(u, v(a_u)) \in H(f): (-w, 1) \cdot (u, v(a_u)) = m(w)\} \subseteq H(f).
\]
\end{defn}

Intuitively, $m(w)$ denotes the minimum valuation achieved assuming that $v(x) = w$,  among the terms of $f$. Then $\vertt_w(f)$ records the corresponding terms of $f$ with the minimum valuation, again assuming that $v(x) = w$.

The following is the power-series analogue of a well-known result for polynomials; the original result for polynomials is first recorded in an unpublished manuscript by Kapranov, and a proof of this lemma for power series can be found in \cite{Rab12}, Lemma 8.4; also see, for example \cite{Mac13}, Theorem 3.1.3. This gives a useful method to computing $\Trop(f)$.
\begin{lemma}
\label{L: Kapranov}
\[
\Trop(f) = \overline{\{w \in \Q_{\geq 0}^d: \# \vertt_w(f) > 1\}}.
\]

\end{lemma}


\subsection{Tropical intersection theory and Newton polygons}
\label{S: Newton}

In this section, we take $d$ power series in $d$ variables in $K \langle U_{P_m} \rangle$ that have finitely many common zeros. We explain that in order to bound the number of common zeros of the $d$ power series, it suffices to know their tropicalizations and their Newton polygons. Since the tropicalizations and the Newton polygons depend only on finitely many terms of the power series convergent on $U_{P_m}$, this section shows that one can approximate a power series of several variables by a polynomial for the purposes of intersection theory. In a sense, this is a stronger approximation than what Weierstrass preparation theorem can tell us; Weierstrass preparation for multivariate power series approximates $f \in K[[t_1, \ldots, t_d]]$ by $f' \in K[t_1][[t_2, \ldots, t_d]]$, whereas here, we approximate $f$ by $f'' \in K[t_1, \ldots, t_d]$.

Let $f \in K \langle U_{P_m} \rangle$. Write $f = \sum a_u x^u$. Define
\[
\vertt_{P_m}(f) := \bigcup_{w \in P_m}\vertt_w(f).
\]

It turns out that $\vertt_{P_m}(f)$ is finite:

\begin{lemma}[\cite{Rab12}, Lemma 8.2]
Let $f \in K \langle U_{P_m} \rangle$ be nonzero. Then $\vertt_{P_m}(f)$ is finite.
\end{lemma}

 This lemma, combined with Lemma \ref{L: Kapranov}, tells us that $\Trop(f)$ determined by only finitely many terms of $f$. 
 
 Now, the following lemma shows that if the coefficients of a power series $f$ are perturbed in a way so that their $v$-adic valuations do not change, and so that $\vertt_P(f)$ does not change, then the tropicalization also stays the same.
 
\begin{lemma}
\label{L: samevert}
Let $f, f' \in K \langle U_{P_m} \rangle$, with $f = \sum_{u}a_ux^u$ and $f' = \sum_u a_u' x^u$. Suppose that $\vertt_{P_m}(f) = \vertt_{P_m}(f')$. Then $\Trop(f) = \Trop(f')$.
\end{lemma}
\begin{proof}
Fix $w \in P_m$. We claim that $\vertt_w(f) = \vertt_w(f')$ for each such $w$. Choose $u_0 \in \Z_{\geq 0}^d$ minimizing $v(a_u) + \langle u,w \rangle$. This means
\[
m_f(w) = v(a_{u_0}) + \langle u_0,w \rangle.
\]
Thus, $(u_0, v(a_{u_0})) \in \vertt_w(f) \subset \vertt_P(f) = \vertt_P(f')$.
So $(u_0, v(a_{u_0})) \in \vertt_{w'}(f)$ for some $w' \in P_m$. In particular, $(u_0, v(a_{u_0})) = (u_0, v(a_{u_0}'))$. Thus,
\[
\min_{u \in \Z_{\geq 0}^d}(v(a_u) + \langle u, w \rangle) = v(a_{u_0}) + \langle u,w \rangle = v(a_{u_0}') + \langle u_0, w \rangle \geq \min_{u \in \Z_{\geq 0}^d}(v(a_u') + \langle u, w \rangle).
\]
The symmetric argument proves the inequality in the other direction, showing that $\vertt_w(f) = \vertt_w(f')$ for each $w \in P_m$. Then by Lemma \ref{L: Kapranov}, $\Trop(f) = \Trop(f')$.
\end{proof}

Given $f \in K \langle U_{P_m} \rangle$, we can find a polynomial $g \in K \langle U_{P_m} \rangle$ such that $\vertt_{P_m}f = \vertt_{P_m} g$, in which case Lemma \ref{L: samevert} implies that $\Trop(f) = \Trop(g)$:

\begin{cor}
\label{C: PStoP}
Let $f  \in K \langle U_{P_m} \rangle$, with $f = \sum_{u}a_ux^u$. Let $\pi: \Z_{\geq 0}^d \times \R \to \Z_{\geq 0}^d$ denote the projection map forgetting the last coordinate. Let $S \subseteq \Z_{\geq 0}^d$ be a finite set containing $\pi(\vertt_{P_m}(f))$. Define the \textbf{auxiliary polynomial of the power series $f$ with respect to $S$} by
\[
g_S = \sum_{u \in S}a_ux^u \in K \langle U_P \rangle.
\]
Then $\Trop(f) = \Trop(g_S)$.
\end{cor}
\begin{proof}
Since $\vertt_{P_m}(f) = \vertt_{P_m}(g_S)$, and since $\Trop(f)$ only depends on $\vertt_P(f)$, the conclusion follows.
\end{proof}

We show that the terms of a power series $f \in K \langle U_{P_m}\rangle$ contained in $\vertt_{P_m}(f)$ also contains all the information about the valuations of zeros of $f$ counting multiplicity. That is, the information about zeros of power series depends on only finitely many terms of $f$ too, and hence the information about the intersection theory of the power series also depends on finitely many terms of $f$.

If $S$ is a finite set of points in Euclidean space, its convex hull is denoted $\conv(S)$.

\begin{defn}
Let $f \in K \langle U_{P_m} \rangle$. For each $w \in \R_{\geq 0}^d$, define its associated Newton polytope
\[
\gamma_w(f) = {\gamma}_w = \pi(\conv(\vertt_w(f))).
\]
\end{defn}

\begin{remark}
The $g_S$ from Corollary \ref{C: PStoP} are chosen so that the $\gamma_w(f) = \gamma_w(g_S)$ as well.
\end{remark}

It turns out that as long as $\bigcap_i V(f_i)$ is finite, the information about the common roots of $f_i \in K \langle U_P \rangle$ having a specified valuation $w$ is encoded in ${\gamma}_w$, as explained below.

\begin{defn}
Let $f_1, \ldots, f_d \in K \langle U_{P_m} \rangle$, $Y_i = V(f_i)$, and $Y = \bigcap_{i} Y_i$. Assume that $Y$ is $0$-dimensional. Then the \textbf{intersection multiplicity} of $Y_1, \ldots, Y_d$ at $w \in \Q^d$ is defined as
\[
i(w; Y_1 \cdots Y_d) := \dim_K H^0(Y \cap U_{\{w\}}, \calO_{Y \cap U_{\{w\}}}),
\]
where $U_{\{w\}} := \trop^{-1}(w)$. In simpler terms, this intersection multiplicity at $w$ is the number of common zeros of the $f_i$ that have the same coordinate-wise valuation as $w$, counting with multiplicity. Since the $U_{\{w\}}$ are disjoint, the \textbf{intersection number} of $Y_1, \ldots, Y_d$ is then
\[
i(Y_1, \ldots, Y_d) := \dim H^0(Y, \calO_Y).
\]
\end{defn}

\begin{defn}
Let $Q_1, \ldots, Q_d$ be bounded polytopes. Define a function
\[
V_{Q_1, \ldots, Q_d}(\lambda_1, \ldots, \lambda_d) := \vol(\lambda_1Q_1 + \cdots + \lambda_d Q_d)
\]
where $+$ denotes the Minkowski sum. The \textbf{mixed volume} of the $Q_i$, denoted $MV(Q_1, \ldots, Q_d)$, is defined as the coefficient of the $\lambda_1 \cdots \lambda_d$-term of $V_{Q_1, \ldots, Q_d}(\lambda_1, \ldots, \lambda_d)$.
\end{defn}

\begin{thm}{\cite{Rab12}*{Theorem~11.7}}
\label{T: tropintersection}
Let $f_1, \ldots, f_d \in K \langle U_{P_m} \rangle$ have finitely many common zeros, and let $w \in \bigcap_{i=1}^d \Trop(f_i)$ be an isolated point in the interior of $P$. For $i = 1, \ldots, d$ let $Y_i = V(f_i)$ and let $\gamma_i = \gamma_w(f_i)$. Then
\[
i(w, Y_1 \cdots Y_d) = MV(\gamma_1, \ldots, \gamma_d).
\]
\end{thm}

In particular, Theorem \ref{T: tropintersection} implies that considering the auxiliary polynomials suffices, as the $\gamma_i$ are determined by only finitely many terms in the $f_i$. 

\begin{thm}
\label{T: approximation}
Suppose that $f_1, \ldots, f_d \in K \langle U_{P_m} \rangle$, and let $g_i$ be the auxiliary polynomials of the $f_i$ with respect to some finite set $S \subseteq M$ containing all $u$ such that $(u, v(a_u)) \in \vertt_{P_m}(f)$. 
Then
\[
\sum_{w \in P^{\circ}} i(w, V(f_1) \cdots V(f_d)) = \sum_{w \in P^{\circ}} i(w, V(g_1) \cdots V(g_d)),
\]
if all the summands on both sides are finite.
\end{thm}
\begin{proof}
By the choice of the $g_i$, $\Trop(f_i) = \Trop(g_i)$ for each $1 \leq i \leq d$, and $\gamma_w(f_i) = \gamma_w(g_i)$ for each $w \in P$. 
We note from Theorem \ref{T: tropintersection} that $\Trop(f_i), \Trop(g_i), \gamma_w(f_i)$ and $\gamma_w(g_i)$ are the only information required in computing the intersection multiplicities.
\end{proof}

The following results for polynomials are useful in estimating the number of zeros of power series. First, a definition:

\begin{defn}
Let $f = \sum_{u \in \Lambda}a_ux^u$ be a polynomial, where $\Lambda \subset \Z^d$ is a finite set. Then the \textbf{Newton polygon} of $f$ is given by
\[
\New(f) = \conv(\{u: u \in \Lambda, a_u \neq 0\}) \subseteq \R^d.
\]
\end{defn}

We recall Bernstein's theorem:

\begin{thm} [\cite{Ber75}]
Let $f_1, \ldots, f_d \in K[x_1, \ldots, x_d]$ be polynomials with finitely many common zeros. Then the number of common zeros of the $f_i$ with multiplicity in $(K^{\times})^d$ is
\[
\MV(\New(f_1), \ldots, \New(f_d)).
\]
\end{thm}

Further, suppose that the $f_i$ have finitely many common zeros whose valuations belong to $P$, and also suppose that the $g_i$ have finitely many common zeros in $K^d$. By Theorem~\ref{T: approximation},
\begin{align*}
&\textup{number of common zeros of the } f_i \textup{ with valuations in } P \\
&= \textup{number of common zeros of the } g_i \textup{ with valuations in  }P\\
&\leq \textup{number of common zeros of the } g_i \textup{ in } (K^{\times})^d\\ &=\MV(\New(g_1), \ldots, \New(g_d)),
\end{align*}
where the last inequality follows by Bernstein's theorem. Thus, we conclude:

\begin{thm}
\label{T: estimation}
Suppose that $f_1, \ldots, f_d \in K \langle U_{P_m} \rangle$, and let $g_i$ be the associated auxiliary polynomials of the $f_i$ with respect to some finite set $S \subseteq M$ containing all $u$ such that $(u, v(a_u)) \in \vertt_P(f)$. Suppose further that $\bigcap_{i=1}^d V(f_i)< \infty$ and $\bigcap_{i=1}^dV(g_i) < \infty$. Then
\[
\# \left( (K^{\times})^d \cap \bigcap_{i=1}^d V(f_i) \right) \leq MV(\New(f_1), \ldots, \New(f_d)).
\]
\end{thm}
\begin{proof}
Again, the proof follows from the fact that by the choice of the $g_i$, we have $\Trop(f_i) = \Trop(g_i)$ and $\gamma_w(f_i) = \gamma_w(g_i)$ for each $1 \leq i \leq d$ and $w \in P$. 
\end{proof}

\begin{remark}
In order to count all solutions of the form $(x_1, \ldots, x_d)$, where some of the $x_i$ may be $0$, one needs to apply Theorem \ref{T: estimation} multiple times, while setting some of the $x_i = 0$.
\end{remark}


\section{Continuity of roots}
\label{S: deformation}

Throughout, $K$ is a complete, algebraically closed valued field with respect to a nontrivial, nonarchimedean valuation $v \colon K^{\times} \to \Q$. \cite{Rab12} studies the intersection theory of power series in $K \langle U_P \rangle$ that have finite intersection. Our goal in this chapter is to analyze the case of the power series in $K \langle U_P \rangle$ have possibly infinite intersection; we will show that these power series have ``small'' deformations that have finite intersection, and that they preserve information about the number of $0$-dimensional components of the original intersection. We make this notion precise, and we obtain an upper bound on the number of $0$-dimensional components (counting multiplicity) of the original intersection, using the new power series with finitely many common zeros, obtained via small $p$-adic deformations.

\subsection{Deformation of power series via rigid analytic geometry and polynomial approximations}

It is known that small deformations do not affect the multiplicity of $0$-dimensional components of intersections in rigid analytic spaces. 

\begin{theorem}[\cite{Rab12}*{Theorem~10.2}, Local continuity of roots] 
\label{T: LocalCont}
Let $A$ be a $K$-affinoid algebra that is a Dedekind domain and let $S = \Sp(A)$. Let $X = \Sp(B)$ be a Cohen-Macaulay affinoid space of dimension $d+1$, let $f_1, \ldots, f_d \in B$, and let $Y \subset X$ be the subspace defined by the ideal $\fraka = (f_1, \ldots, f_d)$. Suppose that we are given a morphism $\alpha: X \to S$ and a point $t \in S$ such that the fibre $Y_t = \alpha^{-1}(t) \cap Y$ has dimension zero. Then there is an affinoid subdomain $U \subset S$ containing $t$ such that $\alpha^{-1}(U) \to U$ is finite and flat.
\end{theorem}

Let $B_K^d := \Sp(K \langle x_1, \ldots, x_d\rangle)$ for $d \geq 1$. The following is immediate from the above theorem, and is applicable to our situation arising from Chabauty's method:

\begin{cor}[\cite{Rab12}*{Example~10.3}]
\label{C: LocalContEx}
Let $X = B_K^d \times B_K^1$ and $S = B_K^1$, with $\alpha \colon X \to S$ the projection onto the second factor. Let $f_1, \ldots, f_d \in K \langle x_1, \ldots, x_d, t \rangle$. If the specializations $f_{1,0}, \ldots, f_{d,0}$ at $t=0$ have only finitely many zeros in $B_K^d$ then there exists $\ve>0$ such that if $|s| < \ve$, then $f_{1,s}, \ldots, f_{d,s}$ have the same number of zeros (counted with multiplicity) in $B_{\kappa(s)}^d$ as $f_{1,0}, \ldots, f_{d,0}$.
\end{cor}
\begin{proof}
Tate algebra in one variable is a Dedekind domain, and $X$ is Cohen-Macaulay (Remark \ref{R: CM}).
\end{proof}

\begin{defn}
 Suppose that $f_1, \ldots, f_d \in K \langle U_{P_m} \rangle$ are power series such that $Y:=V(f_1, \ldots, f_d)$ is possibly infinite. Define
$N_0(f_1, \ldots, f_d)$ to be the number of $0$-dimensional components of $Y$, counting multiplicity.
If $Y$ is finite, then we drop the subscript $0$ to signal its finiteness:
\[
N(f_1, \ldots, f_d) := N_0(f_1, \ldots, f_d) = \dim H^0(Y, \calO_Y).
\]
Also, let $N_0^{\times}(f_1, \ldots, f_d)$ denote the number of $0$-dimensional components of $Y$, whose coordinates are in $\overline{K}^{\times}$.
\end{defn}

\begin{defn}
Let $f = \sum_{u}a_ux^u\in K \langle U_{P_m}\rangle$, and define
\begin{align*}
M(f) &:= \textup{the set of monomials with nonzero coefficients appearing in $f$}\\
&= \{x^u: a_u \neq 0\}.
\end{align*}
Call $f$ \textbf{nondegenerate} if for every $i$ there exists an integer $n > 0$ with $x_i^n \in M(f)$.
\end{defn}

\begin{remark}
\label{R: termsofg}
Any pure power series arising from Chabauty's method is nondegenerate.
\end{remark}

We will need to impose the nondegeneracy conditions in all power series $f$ in order to be able to carry out deformations. Now we prove a series of deformation results for non-stable intersections.

\begin{lemma}
\label{L: nonvanishing}
Let $f \in K\langle U_{P_m} \rangle$ be a nondegenerate power series, and let $q_1, q_2, \ldots, q_n \in U_{P_m}$ such that $q_i \neq 0$ in $K^d$. Then there exists a polynomial $h$ such that $h$ does not vanish on any of $q_1, \ldots, q_{\ell}$ and $M(h) \subseteq M(f)$.
\end{lemma}
\begin{proof}
We will prove by induction on $\ell$ that there exists $g$ such that $M(g) \subseteq M(f)$ and $g(q_1), \ldots, g(q_{\ell}) \neq 0$. The statement is clear when $\ell = 0$, so assume that there is a polynomial $g$ with $M(g) \subseteq M(f)$ satisfying $g(q_1), \ldots, g(q_{\ell}) \neq 0$ and $g(q_{\ell+1}) = \cdots = g(q_n) = 0$, after possibly reordering the $q_i$. If $\ell = n$, then we are done, so assume otherwise. We will show that there exists another polynomial $g' \in M(f)$ such that $g'$ does not vanish on at least $\ell+1$ of the points $q_i$.

Choose a monomial $m \in M(f)$ such that $m(q_{\ell+1}) \neq 0$; such $m$ exists due to the nondegeneracy condition on $f$. We may choose $c \in K^{\times}$ such that $v(cm(q_i)) > v(g(q_i))$ for $1 \leq i \leq \ell$.

Then $g' := g + cm$ satisfies the property that $g'(q_1) \neq 0, \ldots, g'(q_{\ell}) \neq 0$. Also, $g'(q_{\ell+1}) = cm(q_{\ell+1}) \neq 0$. The lemma then follows from an inductive argument on $\ell$.
\end{proof}

\begin{prop}
\label{P: PStoPS}
\hfill
\begin{itemize}
\item [(a)] Suppose that $f_1, \ldots, f_d \in K \langle U_{P_m} \rangle$ are nondegenerate power series. Then there exist nondegenerate $g_1, \ldots, g_d \in K \langle U_{P_m} \rangle$ with $\Trop(f_i) = \Trop(g_i)$ and $\gamma_w(f_i) = \gamma_w(g_i)$, for $1 \leq i \leq d$ and $w \in P_m$, with 
$$N_0(f_1, \ldots, f_d) \leq N(g_1, \ldots, g_d).$$
\item[(b)] Moreover, if the $f_i$ are polynomials, then the $g_i$ may be chosen to be polynomials.
\end{itemize}
\end{prop}

\begin{proof}
We will deform the $f_i$ to the $g_i$ one by one. Specifically, we will prove by induction on $r$ that there exist $g_1, \ldots, g_r$ such that $\Trop(f_i) = \Trop(g_i)$ and $\gamma_w(f_i) = \gamma_w(g_i)$ for $i \in \{1, \ldots, r\}$ and $w \in P$, satisfying $\codim \bigcap_{i=1}^r V(g_i) \geq r$ for each $r \in \{1, \ldots, d\}$, and
\[
N_0(f_1, \ldots, f_d) \leq N_0(g_1, \ldots, g_r, f_{r+1}, \ldots, f_d).
\]

When $r=1$, the statement above is clear, by taking $f_1 = g_1$. Now we prove the statement for $r+1$. Let $C_1, \ldots, C_{\ell}$ be the codimension $r$ irreducible components of $\bigcap_{i=1}^r V(g_i)$. Choose points $P_i \in C_i$, such that $P_i \neq 0$ in $K^d$. We will deform $f_{r+1}$ to $g_{r+1}$ so that $g_{r+1}(P_{i}) \neq 0$ for each $i$, while keeping $\Trop(f_{r+1}) = \Trop(g_{r+1})$ and $\gamma_w(f_{r+1}) = \gamma_w(g_{r+1})$. This will guarantee that $\codim \left(\bigcap_{i=1}^{r+1}V(g_{r+1})\right) \geq r+1$.

From the nondegeneracy assumption of $f_{r+1}$, we have that $V(M(f_{r+1})) = \emptyset$ (if $M(f_{r+1})$ contains $1$) or $V(M(f_{r+1})) = \{0\}$ (if $M(f_{r+1})$ does not contain $1$). In the first case, we may adjust only the constant term to get the desired $g_{r+1}$. Thus, we may assume that we are in the second case. 
Then by Lemma \ref{L: nonvanishing}, we may pick a polynomial $h$ such that $M(h) \subseteq M(f_{r+1})$ such that $h$ that does not vanish on any $P_i$. For small enough nonzero $\ve$, the deformation $f_{r+1} \mapsto f_{r+1} + \ve h =: g_{r+1}$ does not vanish on any of the $P_i$; in this case, the intersection $\bigcap_{i=1}^{r+1}V(g_i)$ has codimension $r+1$, as required. Further, since $h \in M(f_{r+1})$, after possibly making $\ve$ even smaller, both the tropicalization and the $\gamma_w$ of $f_{r+1}$ are identical to those of $g_{r+1}$.

Now we will prove that $N_0(g_1, \ldots, g_d) \geq N_0(f_1, \ldots, f_d)$. It suffices to show that 
\[
N_0(g_1, \ldots, g_{r+1}, f_{r+2}, \ldots, f_d) \geq N_0(g_1, \ldots, g_r, f_{r+1}, \ldots, f_d)
\]
for $r \in \{0, \ldots, d-1\}$; if $r = 0$, then the previous inequality will be interpreted as 
\[
N_0(g_1, f_2, \ldots, f_d) \geq N_0(f_1, \ldots, f_d).
\] 
Let $I$ be the ideal that cut out the dimension $\geq 1$ components of $V(g_1, \ldots, g_r, f_{r+1}, \ldots, f_d)$ in $K \langle U_P\rangle$ and let $\pp_1, \ldots, \pp_{\ell}$ denote the maximal ideals corresponding to the $0$-dimensional components of $V(g_1, \ldots, g_r, f_{r+1}, \ldots, f_d)$. Choose a $f \in I$ such that $f \not\in \pp_i$ for $1 \leq i \leq \ell$. Such choice is possible by the prime avoidance theorem, see for example \cite{AtiMac69}*{Proposition~1.11}. Now we apply Theorem \ref{T: LocalCont} on $\Sp B$, where $B = K \langle U_P \rangle_f$, which states that a small deformation of $V(g_1, \ldots, g_r, f_{r+1}, \ldots, f_d)$ preserves all $0$-dimensional components away from the positive-dimensional locus. This proves the inequality at the beginning of this paragraph, and consequently part (a) of the proposition.

Part (b) follows, since we deform the the $f_i$ by monomials that already appear in $f_i$.
\end{proof}

\subsection{Explicit computation of the upper bound}

Now we consider a residue disk $\calU \subseteq (\Sym^dX)(\C_p)$ whose points reduce to a given $\calQ \in (\Sym^dX)(\F_p)$. Recall from Proposition \ref{P: PureChabauty} that Chabauty's method on $\calU$ yields $d$ pure power series $f_1, \ldots, f_d$ in $d$ variables, whose common zeros in $\C_p^d$ with valuations at least $1/d$ correspond to a set containing the points in $j^{-1}(j(\Sym^dX)(\Q_p)) \cap \overline{J(\Q)})$. Using the results of the previous section, we will obtain an explicit upper bound on the number of common zeros of the $f_i$ in this section by estimating $\New(f_i)$. 
The methods used in this section are reminiscent of \cite{Col85}.

\begin{defn}
\label{D: delta}
For $\ve \in (0, \frac{1}{d})$, $k \in \Z_{\geq 0}$ and $d, \ell \in \Z_{\geq 1}$ with $d \geq \ell$, let
\[
\delta_{\ve}(k,v, \ell) :=  \max \left\{ N \in \Z_{\geq 0} : v(k+N) \geq (\frac{1}{\ell}-\ve)N + v(k)\right\}.
\]
\end{defn}

\begin{remark}
We note that $\delta_{\ve}(k,v,\ell)$ is well-defined, independent of the choice of $\ve$; $v(k+N) = O(\log N)$ as $N \to \infty$, while $(\frac{1}{\ell}-\ve)(N+1)$ increases linearly with $N$.
\end{remark}

\begin{notation}
Given $f  \in W(\F_q)[[t]]$, we mean by $\bar{f}$ the image of $f$ under the natural reduction map of the coefficients $W(\F_q)[[t]] \to \F_q[[t]]$. We will denote $\ord_0(f) := \ord_{t = 0}(\bar{f})$, the exponent of the first term that does not vanish under the reduction map.
\end{notation}

\begin{lemma}
\label{L: truncation}
For any $\ve \in \Q$ satisfying $0 < \ve < \frac{1}{d}$, the following holds: Let $f  \in W(\F_q)\left[\frac 1 p\right][[t]]$ be such that its derivative $f'$ is in $W(\F_q)[[t]]$, and $\ord_0 \bar{f}' = k-1$ for some $k \geq 1$. Let 
\[
F(t_1, \ldots, t_{\ell}) := f(t_1) + \cdots + f(t_{\ell}) = \sum_{u \in \Z_{\geq0}^{\ell}} a_ut^u,
\]
where $t^u$ denotes $t_1^{u_1} \cdots t_{\ell}^{u_{\ell}}$. Let $w \in P_{\ve} = [\ve, \infty)^{\ell}$. If $u = (u_1, \ldots, u_{\ell}) \in \Z_{\geq 0}^{\ell}$ satisfies $u_i > k +  \delta_{\ve}(k,v, \ell)$ for some $1 \leq i \leq \ell$, then $(u, v(a_u)) \notin \vertt_w(F)$.
\end{lemma}

\begin{proof}
Fix $w \in [\ve, \infty)^{\ell}$. Since $F$ is pure, it suffices to consider $u \in \Z_{\geq 0}^{\ell}$ such that $u_1 > k + \delta_{\ve}(k,v,\ell)$ and $u_2 = u_3 = \cdots = u_d = 0$. We will show that there exists $u' \in \Z_{\geq 0}^{\ell}$ such that
\[
v(a_{u'}) + \langle u', w \rangle < v(a_u) +  \langle u,w \rangle.
\]
Then by the definition of $\vertt_w(F)$, the conclusion would follow.

Write $f'(t) = \sum_{i \geq 0} c_it^i$, so $f(t) = \sum_{i \geq 0}\frac{c_i}{i+1}t^{i+1}$. Then $c_i \in \Z_p$ since $f' \in \Z_p[[t]]$, and furthermore, $v(c_{k-1}) = 0$, with $v(c_j) > 0$ for $1 \leq j \leq k-1$, since $\ord_0 \bar{f}' = k-1$.

Then $a_ut^u = \frac{c^{m}}{m+1}t_1^{m+1}$, where $m > k + \delta_{\ve}(k, v, \ell)$. We claim that $u' = (k, 0, \ldots, 0)$ suffices. For any $w  \in [\ve, \infty)^{\ell}$, consider
\begin{align*}
m(w)&:= \min_{u'' \in S_{\sigma}}\{v(a_{u''})+\langle u'', w \rangle\}\\
& \leq v\left(\frac{c_{k-1}}{k}\right)+\langle (k,0, \ldots, 0), (w_1, w_2, \ldots, w_{\ell})\rangle\\
&= v \left(\frac{c_{k-1}}{k}\right) + kw_1
\end{align*}

Since $m > k + \delta_{\ve}(k, v, \ell)$, we have
\[
v(m+1) < (m+1-k)w_1 + v(k),
\]
which rearranges to
\[
-v(k) + kw_1 < -v(m+1) + (m+1)w_1.
\]
Using $v(c_{k-1}) = 0$ and $v(c_m) \geq 0$, this inequality becomes
\[
v\left( \frac{c_{k-1}}{k}\right) + kw_1 < v \left(\frac{c_m}{m+1}\right) + (m+1)w_1.
\]
That is, $(u, v(a_u)) \not\in \vertt_w(F)$, as required.
\end{proof}

\begin{remark}
\label{R: SimpleNew}
Lemma \ref{L: truncation} shows that any pure power series as in the statement of the lemma can be approximated by polynomials whose terms are pure, and whose degree is less than $k + \delta_{\ve}(k, v, \ell)$. This, in turn, means that the Newton polygons of these polynomials are at worst the convex hull of the points $\{(0, \ldots, 0)\} \cup \{(k + \delta_{\ve}(k, v, \ell))e_i : 1 \leq i \leq \ell\} $, where the $e_i$ denotes the $i$-th standard vector. Thus, the Newton polygon can be approximated by a simplex.
\end{remark}

\begin{defn}
Let $A = (a_{ij})$ be a $d \times d$ matrix. The \textbf{permanent} of $A$ is
\[
\Per(A) = \sum_{\sigma \in S_d} \prod_{i=1}^d (a_{i \sigma(i)}).
\]
\end{defn}

\begin{lemma}
\label{L: MV}
Let $A = (a_{ij})$ be a $d \times d$ matrix of positive real numbers, and define the polytopes $X_i \subseteq \R^d$ for $1 \leq i \leq d$ by the following:
\[
X_i = \conv(0, a_{i,1}e_1, \ldots, a_{i,d}e_d)
\]
Then 
\[
\MV(X_1, \ldots, X_d) = \frac{1}{d!} \Per(A).
\]
\end{lemma}
\begin{proof}
The mixed volume is 
\begin{align*}
& \textup{coefficient of $\lambda_1 \cdots \lambda_d$ of }
\vol\left(\conv(0, (\lambda_1 a_{11} + \cdots + \lambda_da_{d1})e_1, \ldots, (\lambda_1 a_{1d} + \cdots + \lambda_da_{dd})e_d)\right) \\
&= \textup{coefficient of $\lambda_1 \cdots \lambda_d$ of } \frac{1}{d!}(\lambda_1 a_{11} + \cdots + \lambda_da_{d1}) \cdots (\lambda_1 a_{1d} + \cdots + \lambda_da_{dd})\\
& = \frac{1}{d!}\sum_{\sigma \in S_d}\prod_{i=1}^d a_{\sigma(i),i} = \frac{1}{d!} \Per(A). \qedhere
\end{align*}
\end{proof}

\begin{lemma}
\label{L: BoundingNewtonPolygon}
Let $f_i \in K[[t_i]]$. Suppose further that $f_i$ converges when $v(t_i) \geq 1/d_i$ and that $f_i'(t_i) = \sum_{j=0}^{\infty} c_{ij}t_i^j \in R[[t_i]]$ for all $i$, where $R$ is the ring of integers for $K$. Suppose also that for each $i$ there exists $k_i \in \Z_{\geq 0}$ such that the coefficients $c_{ij}$ satisfy $v(c_{ij}) > 0$ for $j < k_i$ and $v(c_{ik_i}) = 0$.  From these data, define a multivariate pure power series
\[
F(t_1, \ldots, t_n) := f_1(t_1) + \cdots + f_d(t_n).
\]
Then the Newton polygon of the pure power series $F \in K \langle U_{P_m} \rangle$ (where $m = (1/m_1, \ldots, 1/m_n)$) is contained in the $d$-dimensional simplex defined by the convex hull of the vectors
\[
(k_i + \delta_{\ve}(k_i, v, m_i))e_i,
\]
where $\ve \in \Q$ satisfies $\ve \leq 1/m_i$ for all $i$, and $e_i$ is the $i$-th standard vector.
\end{lemma}
\begin{proof}
Straightforward application of Lemma \ref{L: truncation} and Remark \ref{R: SimpleNew} to each $f_i$ that show up in the pure power series.
\end{proof}

Let $\calQ = \{Q_1, \ldots, Q_d\} \in (\Sym^dX)(\F_p)$. Let $\calU$ be the residue disk of $(\Sym^dX)(\C_p)$ reducing to $\{Q_1, \ldots, Q_d\}$. Decompose the multiset $\{Q_1, \ldots, Q_d\}$ into disjoint multisets $\calS_1, \ldots, \calS_r$ each consisting of a single point with multiplicity $s_j = \#\calS_j$. 

Let $L_j$ be the degree-$s_j$ unramified extension of $K_j$, the field of definition of the points in $\calS_j$, and let $R_j$ be the ring of integers of the $L_j$. For $1 \leq i \leq d$, let $f_{i,j} \in L_j[[t_j]]$ be the power series obtained from Chabauty's method, applied to the residue disk in $(\Sym^{s_j}X)(K_j)$ above the point $\calS_j$, such that their derivatives $f_{i,j}'$ are in $R_j[[t_j]]$. Let
\[
F_i(t_1, \ldots, t_d) = f_{i,1}(t_1) + \cdots + f_{i,d}(t_d),
\]
and let $k_{ij} = \ord_0(f_{ij})$.

Then define the $d \times d$ matrix $A_{\calP} = (a_{ij})$ by $a_{ij} = k_{i,j} + \delta_{\ve}(k_{i,j}, v, s_i)$ for each residue disk and suitably small $\ve$. 

\begin{thm} 
Keep the notation from the previous paragraph. Then the $F_i$ satisfy
\[
N_0^{\times}(F_1, \ldots, F_d) \leq \frac{1}{d!} \Per(A).
\]
\end{thm}
\begin{proof}
By Proposition \ref{P: PStoPS}, we may as well assume that the power series that we get from Chabauty's method have finitely many common zeros (that is, a deformation of the power series exists, such that the tropicalizations and the $\gamma_w$ stay constant). This means that, by Theorem \ref{T: estimation}, that the number of isolated solutions can be written as the mixed volume of Newton polygons. Now combine Lemma \ref{L: MV} and Lemma \ref{L: BoundingNewtonPolygon}. 
\end{proof}

Recall that $N_0^{\times}(F_1, \ldots, F_d)$ counts the $0$-dimensional components of the common zeros of the $F_i$ in $(\C_p^{\times})^d$. Thus, we need to count the solutions in which some of the coordinates are $0$ separately. For example, if we wish to count the solutions that are of the form $(\C_p^{\times})^{(d-1)} \times \{0\}$, it suffices to consider
\[
N_0^{\times}(F_1(t_1, t_2, \ldots, t_{d-1}, 0), \ldots, F_{d-1}(t_1, t_2, \ldots, t_{d-1}, 0)),
\]
which is bounded above by $\frac{1}{(d-1)!}\Per(B)$, where $B$ is a $(d-1) \times (d-1)$ minor of $A$ that takes the first $(d-1)$ rows and columns.
Thus, let
\[
\Per(A)' := \sum_{0 \leq i \leq d} \sum_{j \in \Lambda_i} \frac{1}{i!}\Per(A_{ij}),
\]
where $A_{ij}$ denotes the $i \times i$ minor of $A$ that takes the first $i$ columns (and any $i$ rows), and $A_{00}$ is the $0 \times 0$ matrix whose permanent is understood to be $1$ (since if (0, \ldots, 0) were a solution to the $F_i$, it would contribute at most $1$ to $N_0(F_1, \ldots, F_d)$).

\begin{thm}
\label{T: bound}
Suppose $X$ is a nice curve over $\Q$ with good reduction at $p$ satisfying Assumption \ref{A: assumption}, and let $\omega_1, \ldots, \omega_d \in H^0(X_{\Q_p}, \Omega^1)$ be independent differential forms that vanish on $\overline{J(\Q)}$ such that $\bar{\omega}_i \neq 0$. Then keeping the notation as above, with the $k_{i,j}$ corresponding to the order of vanishing of $\omega_i$ at the point $P_j$, the number of points outside of the special set of $(\Sym^dX)(\Q_p)$ is at most
\[
\sum_{\calP \in (\Sym^dX)(\F_p)} \frac{1}{N_{\calP}}\Per(A_{\calP})'
\]
\end{thm}
\begin{proof}
Apply the above theorem to each residue disk of $(\Sym^dX)(\Q_p)$, and use Corollary \ref{cor: comparison}. The $\frac{1}{N}$ accounts for the ordering of the solutions, since the order of the points does not matter in $\Sym^dX$.
\end{proof}

The above theorem shows that there is an upper bound on the number of points outside of the special set, depending only on the choice of $g, d$ and $p$. If we bound $\# (\Sym^dX)(\F_p)$ in terms of $g, d$ and $p$, then this would complete the proof of Theorem \ref{T: main}:

\begin{prop}
\label{P: residuedisks}
Given a nice curve $X$ of genus $d$ with good reduction at $p$ and $d \geq 1$,
\[
\#((\Sym^dX)(\F_p)) \leq (1 + 2g p^{d/2} + p^d)^d.
\]
\end{prop}
\begin{proof}
We use the Hasse-Weil bound on $X$, along with the fact that if $\{P_1, \ldots, P_d\} \in (\Sym^dX)(\F_p)$, then $P_i \in X(\F_{p^d})$ for $1 \leq i \leq d$.
\end{proof}

Then the proof of Theorem \ref{T: main} follows by combining the statements of Proposition \ref{P: residuedisks} and Theorem \ref{T: bound}.

\section{An application}
\label{S: application}

In this section, we prove the following corollary:

\begin{cor}
\label{C: calculations}
We can take $N(2,3,3) = 1539$ for any $X/\Q$ a hyperelliptic curve whose affine model $y^2 = f(x)$ satisfies $\deg(f) =7$ (so that $g = 3$), such that $\rank J \leq 1$, and such that $X$ has good reduction at $p=2$. 
\end{cor}

As noted in the introduction, Assumption \ref{A: assumption} is unnecessary when $\rk J \leq 1$. (And we expect that 100\% of hyperelliptic curves have ranks $0$ or $1$, assuming Goldfeld's conjecture!)

\begin{lemma}
Let $X$ be a smooth projective odd hyperelliptic curve of genus $3$ that has good reduction at $2$.  Then $\# (\Sym^2X)(\F_2) \leq 19$.
\end{lemma}
\begin{proof}
We first note that a mod-$2$ reduction of an odd hyperelliptic curve of genus $3$ corresponds to an equation of the form $y^2 + g(x)y = h(x)$, with $g(x), h(x) \in \F_2[x]$, with $\deg g \leq 3, \deg h = 7$. Let $\calP \in (\Sym^2X)(\F_2)$. It can be viewed as a multiset of two points $\calP = \{P_1, P_2\}$. We denote by $x(P_i)$ and $y(P_i)$ the $x$- and $y$-coordinates of $P_i$, respectively, for $i = 1,2$. We have two cases:

\textbf{Case 1}: When $P_1, P_2 \in X(\F_2)$. If $P \in X(\F_2)$, then $x(P), y(P) \in \F_2$ or $x(P) = \infty$, so in particular, one must have $x(P) \in \{0,1, \infty\}$. There are at most two points above each $\F_2$-point in the map $X \to \PP^1$, so there are at most $5$ points in $X(\F_2)$.
Let $\#X(\F_2) = a$. Then there are ${a \choose 2} + a$ points $\calP \in (\Sym^2X)(\F_2)$ of the form $\{P_1, P_2\}$ with $P_i \in X(\F_2)$; the first term counts $\{P_1,P_2\}$ with $P_1$ and $P_2$ distinct, and the second terms counts $\{P_1,P_2\}$ with $P_1=P_2$.

\textbf{Case 2}: When $P_1, P_2 \in X(\F_4) \backslash X(\F_2)$ are Galois conjugates. In this case, there are at most $4$ points of $X(\F_4)-X(\F_2)$ above $\PP^1(F_4)-\PP^1(F_2)$. But there could also be points of $X(\F_4)-X(\F_2)$ above $\PP^1(\F_2)$; the number of these is $5-a$, since all of the $5$ $\Fbar_2$-points of X above $\PP^1(\F_2)$ are either $\F_2$-points or $\F_4$-points. So there are at most $9-a$ such points.

Clearly, the choice of $1 \leq a \leq 5$ that maximizes ${a \choose 2} + 9$ is $a=5$, which means that there are at most $19$ points in $(\Sym^2X)(\F_2)$.
\end{proof}

Now we focus on a single residue disk of $(\Sym^2X)(\F_2)$ and compute the possible number of points on each residue disk.

Since $g=3$, the degree of $\bar{\omega}$ is $2g-2 = 4$. We start by computing $\delta_{\ve}(k, v, \ell)$ in Definition \ref{D: delta} for when $k = 1,2,3,4$. We take $\ve \in (0, \frac{1}{2})$ as small as possible, as that minimizes $\delta_{\ve}(k,v,\ell)$. Then we have
\[
\delta_{\ve}(4,2,2) = 0, \quad \delta_{\ve}(3,2,2) = 5, \quad
\delta_{\ve}(2,2,2) = 2, \quad \delta_{\ve}(1,2,2) = 3.
\]
Thus, for a residue disk over $\calP$, the largest value of $\Per A_{\calP}$ is given from the $2 \times 2$ matrix whose entries are all $k + \delta_{\ve}(k,v, \ell)$ with $k = 3$. That is, the maximal value for $\Per A_{\calP}$ is $128$.

Now, there are two $1 \times 1$ minors that we need to compute, from the definition of $\Per(A)'$ in the previous chapter. Again, the maximal values for these are $8$, obtained when $k=3$. This gives $\Per(A)' \leq \frac{1}{2} \cdot 128 + 8 + 8 + 1 = 81$.

Now, we apply Theorem \ref{T: bound} on the $19$ residue disks with $N_{\calP} \geq 1$ and $\Per(A_{\calP})' \leq 81$. This gives the upper bound of $81 \times 19 = 1539.$ This completes the proof of Corollary \ref{C: calculations}.

\end{document}